\documentclass[11pt]{amsart}
\usepackage{amsmath}
\usepackage{amsthm}
\usepackage{amsfonts} 

\usepackage{mathabx}
\usepackage{amssymb}
\usepackage[english]{babel} 
\usepackage[colorlinks = true,
            linkcolor = blue,
            urlcolor  = blue,
            citecolor = blue,
            anchorcolor = blue]{hyperref} 
            
\usepackage{cleveref}
\usepackage{calc} 
\usepackage{dsfont}
\usepackage{enumitem}
\usepackage{placeins}

\usepackage{graphicx}
\usepackage{subcaption}
\usepackage{float}

\usepackage{physics}
\usepackage{tikz}

\allowdisplaybreaks
\usepackage[a4paper, lmargin=0.1666\paperwidth, rmargin=0.1666\paperwidth, tmargin=0.1111\paperheight, bmargin=0.1111\paperheight]{geometry} 
\usepackage[all]{nowidow}
\usepackage{lipsum} 

\newtheorem{claim}{Claim}

\theoremstyle{plain}
\newtheorem{theorem}{Theorem}[section]

\newtheorem{lemma}[theorem]{Lemma}
\newtheorem{corollary}{Corollary}[section]

\theoremstyle{definition}
\newtheorem{definition}[theorem]{Definition}

\theoremstyle{remark}

\newcommand{\opp}{\operatorname{opp}}

\newtheoremstyle{named}{}{}{\itshape}{}{\bfseries}{.}{.5em}{\thmnote{#3}#1}
\theoremstyle{named}
\newtheorem*{namedtheorem}{}
\setlist[enumerate]{itemsep=0pt, topsep=2pt}

\usepackage{algorithm}
\usepackage{algpseudocode}

\title{Positive braid closures and taut foliations}
\author{Zipei Nie}

\address{Department of Mathematics\\
University of Illinois at Urbana--Champaign\\
Urbana, IL 61801, USA}

\email{znie@illinois.edu}
	
\date{\today}

\begin{document} 

\begin{abstract}
We study taut foliations on the complements of non-split positive braid closures in $S^3$. If $L$ is such a link with components $L_1,\ldots,L_n$ and at least one component is not the unknot, then the Dehn surgery along a multislope $(s_1,\ldots,s_n)\in\mathbb{Q}^n$ satisfying $s_i<2g(L_i)-1$ for $i=1,2,\ldots, n$ yields a non-L-space that admits a co-oriented taut foliation.
\end{abstract}
\maketitle
\section{Introduction}
A central theme in $3$-manifold topology is the relationship between taut foliations, left-orderability, and Heegaard Floer homology. The L-space conjecture \cite{boyer2013spaces,juhasz2015survey} predicts that the following are equivalent for an irreducible rational homology $3$-sphere $Y$:
\begin{enumerate}[label=(\alph*)]
    \item $Y$ is a non-L-space;
    \item $\pi_1(Y)$ is left-orderable;
    \item $Y$ admits a co-oriented taut foliation.
\end{enumerate}
In general, only the implication (c)$\Rightarrow$(a) \cite{ozsvath2004holomorphic,bowden2016approximating,kazez2017c0} is known.

Dehn surgery on knots in $S^3$ provides a natural source of both L-spaces and non-L-spaces. For a nontrivial knot $K\subset S^3$, the set of slopes yielding L-spaces \cite{kronheimer2007monopoles,rasmussen2017floer} is known to be one of the following three possibilities:
\begin{enumerate}
\item $[2g(K)-1,\infty)$, in which case $K$ is called a positive L-space knot;
\item $(-\infty,-2g(K)+1]$, in which case $K$ is called a negative L-space knot;
\item the empty set, in which case $K$ is called a non-L-space knot.
\end{enumerate}
Here $g(K)$ denotes the Seifert genus of $K$. Thus, the L-space interval is completely determined once we know which case $K$ falls into. In contrast, the corresponding interval for taut foliations or left-orderability is much less understood.

In this paper, we study the existence of taut foliations on non-L-spaces obtained by Dehn surgeries on L-space knots in $S^3$.

Most known positive L-space knots in $S^3$ arise as positive braid closures. For example, among the $632$ positive L-space knots in the SnapPy census with at most $9$ tetrahedra, $631$ are positive braid closures \cite{baker2024census}, and the remaining example $o9\_30634$ is the only known exception. Moreover, every positive $(1,1)$ L-space knot in $S^3$ is a positive braid closure \cite{nie2021explicit}. 

When a link is represented by a braid closure, it is not straightforward to determine whether it is a knot or a multi-component link. For this reason, we work in the more general setting of links rather than restricting ourselves to knots. Our main theorem constructs taut foliations on the complement of a non-split positive braid closure in $S^3$ with at least one nontrivial component\footnote{It is not sufficient to merely assume $L$ is nontrivial. The Hopf link provides a counterexample.}.

\begin{namedtheorem}[Main Theorem]
Suppose $L$ is the closure of a non-split positive braid in $S^3$ with components $L_1,\ldots,L_n$, and assume that at least one component is not the unknot. Then the Dehn surgery along a multislope $(s_1,\ldots,s_n)\in\mathbb{Q}^n$ satisfying $s_i<2g(L_i)-1$ for all $i=1,\ldots,n$ yields a $3$-manifold admitting a co-oriented taut foliation.
\end{namedtheorem}

As an immediate corollary, the resulting manifolds are non-L-spaces.

Lyu proved that $(1,1)$ non-L-space knots are persistently foliar \cite{lyu2025persistent}. It is also known that positive $(1,1)$ L-space knots in $S^3$ are positive braid closures, and that the L-spaces obtained by Dehn surgeries on $(1,1)$ knots have non-left-orderable fundamental groups \cite{nie2021explicit,Li2024TautFoliations}. Combining these results with our main theorem yields the following corollary.

\begin{corollary}
For $3$-manifolds obtained by Dehn surgery on a $(1,1)$ knot in $S^3$, the implications
\[
(a)\Leftrightarrow(c)\quad\text{and}\quad (a)\Leftarrow(b)
\]
hold in the L-space conjecture.
\end{corollary}

The remaining implication $(a)\Rightarrow(b)$ is still open, even for the $(-2,3,7)$-pretzel knot \cite{varvarezos2021representations}. 

Our construction differs from previous approaches. Most existing methods begin with a Seifert surface and then apply sutured manifold decompositions and splittings to produce a laminar branched surface. This has been the standard approach for over two decades \cite{roberts2000taut,roberts2001taut,li2003boundary,nakae20072,li2014taut,krishna2020taut,santoro2024spaces,krishna2025taut}.

Instead, we first pinch the Seifert surface to obtain a simpler branched surface, and then perform splittings. From the theory of laminar branched surfaces \cite{li2002laminar}, this intermediate step is not logically necessary. However, for explicit constructions and algorithmic implementation, simplifying the branched surface provides substantial advantages.

In Li's general framework, finding suitable splittings may require searching an exponentially large space of possibilities. In contrast, we design a greedy splitting algorithm of linear complexity (Algorithm~\ref{alg:brick-curve}) that produces a laminar branched surface from the simplified model.

In Krishna's work \cite{krishna2020taut,krishna2025taut} on taut foliations in braid closure complements, the branched surface is constructed by specifying coorientations for the product disks in Gabai's disk decomposition. This is equivalent to choosing curves $\gamma_1,\gamma_2,\ldots,\gamma_k$ in Section~\ref{sec:curve} with the requirement that each $\gamma_i$ lies in the closed region between the $i$-th row and the $(i+1)$-st row. Our construction removes this restriction, which allows greater flexibility and leads to a more general result.

The paper is organized as follows. Section~\ref{sec:2} constructs the branched surface $B$. Section~\ref{sec:3} splits $B$ into a laminar branched surface $B'$, assuming the existence of simple closed curves \(\gamma_1, \gamma_2, \ldots, \gamma_k\) satisfying Properties~\ref{G1}--\ref{G8}. Section~\ref{sec:curve} describes the algorithm producing the curves $\gamma_1,\dots,\gamma_k$ used in the splitting. Section~\ref{sec:5} analyzes the boundary train track. Section~\ref{sec:6} combines these ingredients to prove the main theorem.
\subsection*{Acknowledgements.}
The author is grateful to Qingfeng Lyu for many stimulating and ongoing discussions, and to Jacob Rasmussen for valuable advice and insightful discussions.

\section{Branched surface construction}\label{sec:2}
Let $L$ be a non-split positive braid closure in $S^3$. In this section, we construct a branched surface properly embedded in \(S^{3}\setminus \operatorname{int}(\nu(L)).\)

We use the standard terminology for branched surfaces in a $3$-manifold.

The \emph{branch locus} of a branched surface $B$ is the set of non-manifold points of $B$. Each component of the branch locus carries a well-defined \emph{branch direction}. The complementary components of the branch locus in $B$ are called \emph{branch sectors}.

A branched surface $B\subset S^{3}\setminus \operatorname{int}(\nu(L))$ admits a \emph{fibered neighborhood} $N(B)$ whose fibers are intervals ($\emph{I}$-\emph{fibers}). The boundary of $N(B)$ decomposes into three parts: the \emph{horizontal boundary} $\partial_h N(B)$ (transverse to the $I$-fibers), the \emph{vertical boundary} $\partial_v N(B)$ (tangent to the $I$-fibers), and $N(B)\cap\partial\nu(L)$.

We also use the standard terminology for train tracks. 

Vertices are called \emph{switches}, edges are called \emph{branches}, and immersed curves that follow the switch directions are called \emph{train routes}. For a properly embedded branched surface $B$, the boundary $\partial B$ is a train track on $\partial\nu(L)$.

First, we choose a link diagram $D$ of $L$. Next, we place a train track $\tau$, obtained from a local modification of $D$, in front of $D$. Finally, we place a loop $S^1$ further in front of $\tau$.

The construction of the branched surface $B$ proceeds in two stages. Projecting $\tau$ onto $D$ produces the first part of the branched surface (see Figure~\ref{fig:level12}). Collapsing $\tau$ to the loop $S^1$ and attaching a disk yields the remaining part of the branched surface (see Figure~\ref{fig:level23}).

\begin{figure}[!htbp]
\centering
\includegraphics[width=8cm]{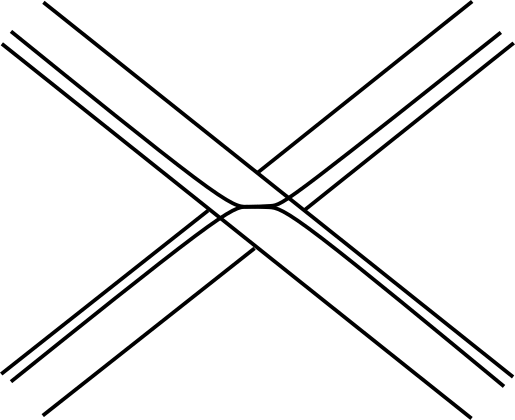}
\caption{$\bigl\{(p,t)\in(\tau\times[0,2])\setminus\operatorname{int}(\nu(L)):\ t\ge h(p)\bigr\}$}
\label{fig:level12}
\end{figure}

\begin{figure}[!htbp]
\centering
\begin{tikzpicture}
\node[anchor=south west, inner sep=0] (img)
{\includegraphics[width=8cm]{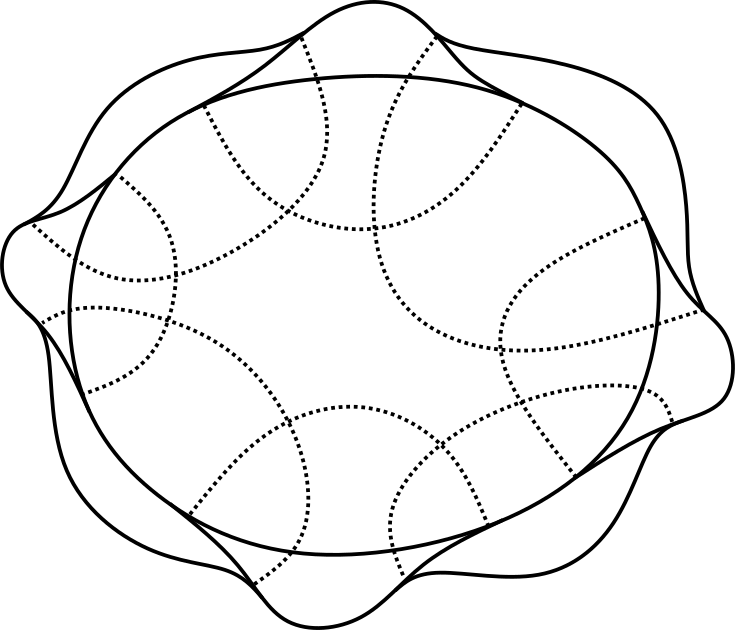}};
\begin{scope}[x={(img.south east)},y={(img.north west)}]
\node at (0.45,0.5) {$D^2$};
\node at (0.8,0.98) {$(\tau \times [2,3])/\sim$};
\end{scope}
\end{tikzpicture}
\caption{$((\tau \times [2,3])/\sim)\cup D^2$.}
\label{fig:level23}
\end{figure}

\subsection{The first level}
Fix a standard positive braid diagram $D\subset S^1\times I$ for $L$, in which the braid is drawn horizontally and all crossings are positive. We assume that the number of braid strands is minimal among all positive braid diagrams representing $L$. We realize the link $L$ in a standard product neighborhood $S^1\times I\times[0,1]\subset S^3$ using the braid diagram $D\subset S^1\times I$.

Choose finitely many points $p$ on $D$, away from the crossings, so that $D$ is decomposed into a finite collection of arcs, and each arc lies entirely either on the over-strand (called \emph{over-arcs}) or on the under-strand (called \emph{under-arcs}) at every crossing it encounters. We assume that the number of chosen points $p$ is minimal among all such choices. Under this assumption, each over–arc intersects at least one under–arc at a crossing, and each under–arc intersects at least one over–arc at a crossing.

The over-arcs are placed in the level $S^1\times I\times\{1\}$,
while the under-arcs are placed in the level $S^1\times I\times\{0\}$. At each chosen point $p$, we connect the endpoints of the adjacent arcs by an interval arc $\{p\}\times[0,1]$.

From now on, we regard $L$ as the union of these arcs.

\subsection{The second level}

We construct a train track $\tau\subset S^1\times I$ that is transverse to every vertical line by a local modification of the positive braid diagram $D\subset S^1\times I$.

Near each crossing of $D$, we replace the crossing with a short horizontal segment, with two branches emerging from the left and two branches emerging from the right, as in the standard smoothing into a train track. Outside a small neighborhood of the crossings, the diagram remains unchanged.

The train track $\tau$ can be decomposed into arcs of three types: \emph{over-type}, if the arc arises from an over-arc of $D$; \emph{under-type}, if it arises from an under-arc of $D$; and \emph{mix-type}, if it is a short horizontal segment.

Define a function $h\colon \tau \to \{0,1\}$ by
\[
h(p) :=
\begin{cases}
1, & \text{if $p$ is on an over-type arc},\\
0, & \text{if $p$ is on an under-type arc},\\
0, & \text{if $p$ is on a mix-type arc}.
\end{cases}
\]

Since the resulting train track $\tau$ differs from $D$ only within a small neighborhood of the crossings, we may assume that $(p, h(p))\in \nu(L)$. 

\subsection{The third level}
We view the product neighborhood $S^1\times I\times[0,1]\subset S^3$ as a submanifold of a larger product region
\[
S^1\times I\times[0,1]\subset S^1\times I\times[0,3]\subset S^3.
\]

Suppose that the complement $(S^1\times I)\setminus \tau$ contains $N$ bigons. Define an equivalence relation $\sim$ on $(S^1\times I)\times[0,3]$ by declaring that $(p,t)\sim(p',t')$ if and only if
\begin{enumerate}[label=(\arabic*)]
\item $(p,t)=(p',t')$, or 
\item there exists a bigon $R_i$ such that
\begin{enumerate}
    \item $p,p'\in R_i$, and
    \item the line segment from $p$ to $p'$ is vertical, and
    \item $t=t'\ge 2+\frac{i}{N+1}$.
\end{enumerate} 
\end{enumerate}

Under the quotient map, the image of the branched surface
\[
\bigl\{ (p,t)\in (\tau\times[0,3])\setminus \operatorname{int}(\nu(L)) : t \ge h(p) \bigr\}\subset S^1\times I\times[0,3]
\]
is also a branched surface. Since the quotient space $S^1\times I\times[0,3]/\sim$ is homeomorphic to the standard copy of $S^1\times I\times[0,3]$ in $S^3$, we may regard the image as a branched surface in $S^1\times I\times[0,3]$.

Because $L$ is non-split, this branched surface intersects the level $S^1\times I\times\{3\}$ in a single loop $S^1$. We attach another disk $D^2$ in $S^3\setminus \operatorname{int}(\nu(L))$ along this loop, and denote the resulting branched surface properly embedded in $S^3\setminus \operatorname{int}(\nu(L))$ by $B$.

\subsection{Properties of the branched surface \texorpdfstring{$B$}{B}}

In this subsection, we prove some properties of the branched surface $B$ constructed above. 

\begin{theorem}\label{cooriented}
    The branched surface $B$ admits a co-orientation.
\end{theorem}
\begin{proof}
By construction, the train track $\tau\subset S^1\times I$ is transverse to every vertical line. We declare the positive side to be the downward vertical direction. The co-orientation descends through the quotient map and extends naturally over the disk $D^2$ attached at the level $S^1\times I\times \{3\}$.
\end{proof}

A \emph{sutured manifold} is a pair $(M,\gamma)$, where $M$ is a compact $3$-manifold with boundary and $\gamma$ is a collection of simple closed curves in $\partial M$ that divide $\partial M$ into two subsurfaces $R_+(\gamma)$ and $R_-(\gamma)$.

\begin{theorem}\label{disk}
The sutured manifold
\[
\bigl( S^3\setminus \operatorname{int}(\nu(L)\cup N(B)), \overline{\partial\nu(L)\setminus  N(B) }\cup \partial_v N(B)\bigr)
\]
is homeomorphic to
\[
(D^2\times I,\ \partial D^2\times I).
\]
\end{theorem}
\begin{proof}
    First, there exists a deformation retraction from
\[
\nu(L)\cup
N\left(
\bigl\{ (p,t)\in (\tau\times[0,3])\setminus \operatorname{int}(\nu(L))
: t \ge h(p) \bigr\}
\right)
\]
onto
\[
\bigl\{ (p,t)\in \tau\times[0,3] : t \ge h(p) \bigr\}.
\]

Next, define a homotopy
\[
f \colon \tau \times [0,3] \times [0,1] \longrightarrow \tau \times [0,3]
\quad\text{by}\quad
f(p,t,x) := \bigl(p,\max(t,3x)\bigr).
\]
This homotopy gives a deformation retraction from
\[
\bigl\{ (p,t)\in \tau\times[0,3] : t \ge h(p) \bigr\}
\]
onto $\tau \times \{3\}$.

Both deformation retractions descend through the quotient map. Therefore, by composing these deformation retractions, passing to the quotient, and finally contracting the attached disk, we conclude that $\nu(L)\cup N(B)$ is contractible. Therefore, \(S^3\setminus \operatorname{int}(\nu(L)\cup N(B))\)
is homeomorphic to $D^2\times I$.

Both $\overline{\partial\nu(L)\setminus  N(B) }$ and $\partial_v N(B)$ are unions of disjoint rectangles. Thus, to prove the desired result, it suffices to show that these rectangles are glued together to form a single annulus, rather than several disjoint annuli. Since $B$ is co-oriented (see \Cref{cooriented}), it suffices to consider the core curves of these annuli.

The core curves can be described by a modification of the train track $\tau\subset S^1\times I$. At each switch of $\tau$, if there are two branches entering from the left, we disconnect the top-left branch from the switch; if there are two branches entering from the right, we disconnect the bottom-right branch from the switch. The resulting arcs represent the core curves of the rectangles in $\overline{\partial\nu(L)\setminus  N(B) }$.

For each bigon complementary region of $\tau$, we connect the two disconnected branches corresponding to the switches at the two vertices of the bigon. These connecting arcs represent the core curves of the rectangles in $\partial_v N(B)$. Let $\gamma$ denote the resulting $1$-manifold. We will prove that $\gamma$ is connected by induction on the number of braid strands.

If $L$ has only one braid strand, then the associated train track $\tau$ has no switches and no bigon complementary regions. In this case, both $\tau$ and the resulting curve $\gamma$ are single embedded loops.

Now suppose that $L$ is a non-split positive braid closure with at least two braid strands. Let $P_1,P_2,\ldots,P_c$ be all crossings $P$ of the braid diagram such that the complementary region directly below $P$ is not a bigon, listed from left to right and indexed cyclically so that $P_{c+1}=P_1$. Perform a $0$-resolution at each of the crossings $P_1,P_2,\ldots,P_c$. Then the link $L$ decomposes as a split union $L'\cup U$, where $L'$ is a non-split positive braid closure with one fewer strand, and $U$ is an unknot passing through the crossings $P_1,\ldots,P_c$.

Let $\gamma'$ denote the $1$-manifold obtained from $L'$ by the above construction. By the induction hypothesis, $\gamma'$ is a single loop. To recover the curve $\gamma$ associated to $L$, we modify $\gamma'$ near the crossings $P_1,\ldots,P_c$ as follows. For each $i=1,\ldots,c$, we disconnect $\gamma'$ in a neighborhood of the crossing $P_{i+1}$, and then reconnect it across the bigon complementary region with vertices $P_i$ and $P_{i+1}$. Concretely, we replace the arc of $\gamma'$ near $P_{i+1}$ by an arc running inside this bigon that connects $P_{i+1}$ to $P_i$, together with an arc of $U$ connecting $P_i$ to $P_{i+1}$.

After performing this operation for all $i=1,\ldots,c$, the loops $\gamma'$ and $U$ are joined into a single loop, which is precisely the $1$-manifold $\gamma$ associated to $L$. This completes the induction.

Therefore, \[
\bigl( S^3\setminus \operatorname{int}(\nu(L)\cup N(B)),
\overline{\partial\nu(L)\setminus N(B) }\cup \partial_v N(B)\bigr)
\]
is homeomorphic to
\[
(D^2\times I,\ \partial D^2\times I)
\] as a sutured manifold.
\end{proof}

\begin{corollary}\label{transversal}
There exists a closed curve in $ S^3\setminus  \operatorname{int}(\nu(L))$ positively transverse to $B$, intersecting every branch sector of $B$.
\end{corollary}
\begin{proof}
    For each branch sector of $B$, choose a point in its interior and take a short arc positively transverse to $B$ through this point, disjoint from $B$ otherwise. Since $(S^3\setminus \operatorname{int}(\nu(L)))\setminus B$ is connected (see \Cref{disk}), we can connect the arcs in order, respecting the orientations, to obtain a closed curve. By construction, this curve is positively transverse to $B$ and intersects every branch sector of $B$.
\end{proof}

\section{Splitting the branched surface}\label{sec:3}
In this section, we assume that at least one component of $L$ is not the unknot. 

We recall the definition of a \emph{laminar branched surface}.

\begin{definition}[{\cite[Definition~2.3]{li2003boundary}}]\label{def-laminar}
Let $B$ be a branched surface properly embedded in a $3$-manifold $M$. We call $B$ a \emph{laminar branched surface} if the following conditions hold after eliminating trivial bubbles:

\begin{enumerate}[label=(\alph*)]
\item the horizontal boundary $\partial_h N(B)$ is incompressible and $\partial$-incompressible in $M\setminus \operatorname{int}(N(B))$, there is no monogon in $M\setminus \operatorname{int}(N(B))$, and no component of $\partial_h N(B)$ is a sphere or a disk properly embedded in $M$;
\item the complement $M\setminus \operatorname{int}(N(B))$ is irreducible, and $\partial M\setminus \operatorname{int}(N(B))$ is incompressible in $M\setminus \operatorname{int}(N(B))$;
\item the branched surface $B$ contains no Reeb branched surface;
\item the branched surface $B$ has no sink disk and no half sink disk.
\end{enumerate}
\end{definition}

For the definitions of \emph{trivial bubble}, \emph{Reeb branched surface},
\emph{sink disk}, and \emph{half sink disk}, see \cite{li2003boundary}.

We will apply the splitting technique (see \cite[Section~5]{li2002laminar}) to obtain a laminar branched surface $B'$ from $B$.

\subsection{Splitting surface construction}
Let $F\subset N(B)$ be a compact embedded surface transverse to the $I$-fibers of $N(B)$. Let $F\times I$ be a product neighborhood of $F$ such that each $x \times I$ ($x\in F$) is a subarc of an $I$-fiber of $N(B)$. After a small perturbation, the complement $N(B)\setminus \operatorname{int}(F\times I)$ is a fibered neighborhood of another branched surface $B'$. We say that $B'$ is obtained from $B$ by splitting along $F$. 

Our construction of the splitting surface $F$ depends on a choice of $k$ simple closed curves \(\gamma_1, \gamma_2, \ldots, \gamma_k \subset S^1 \times I.\) We require the following eight properties of these curves:
\begin{enumerate}[label=(G\arabic*)]
    \item\label{G1} Each curve $\gamma_i$ is transverse to every vertical line and hence intersects each vertical line $\{x\}\times I$ once.
    \item\label{G2} Each curve $\gamma_i$ is composed of two arcs: one arc connects the two vertices of a bigon in the complement of $\tau$, while the other arc lies on the train track $\tau$.
    \item\label{G3} For each vertical line $\{x\}\times I$ with $x\in S^1$, the intersection point
    $\gamma_i \cap (\{x\}\times I)$ lies above or coincides with
    $\gamma_{i+1} \cap (\{x\}\times I)$ for all $i=1,\ldots,k-1$.
    \item\label{G4} The bigons associated to the curves $\gamma_1,\gamma_2, \ldots, \gamma_k$ are pairwise distinct.
    \item\label{G5} There exists a point $p\in \gamma_1$ such that $p$ lies on an under-type arc of $\tau$, and there is no bigon directly above $p$.
    \item\label{G6} There exists a point $p\in \gamma_k$ such that $p$ lies on an under-type arc of $\tau$, and there is no bigon directly below $p$.
    \item\label{G7} For each $1\le i\le k-1$, there exists a vertical segment connecting a point $p_1 \in \gamma_i$ to a point $p_2 \in \gamma_{i+1}$ whose interior is disjoint from $\tau$, such that one of the points $p_1,p_2$ lies on an under-type arc of $\tau$ and the other does not lie on $\tau$.
    \item\label{G8} For each under-arc of $D$, there exists a point $p\in \tau$ lying on an under-type arc of $\tau$ arising from that under-arc, such that $p\notin \gamma_i$ for all $1\le i\le k$.
\end{enumerate}
These properties will be established in Section~\ref{sec:curve}.

Each $\gamma_i$ determines a subsurface
\[
D_i := \bigl(\bigl((\gamma_i \times [1,3])/ \sim\bigr) \cap B \bigr)  \cup D^2
\]
in the branched surface $B$, where the quotient is taken with respect to the equivalence relation $\sim$ used in the construction of $B$, and the disk $D^2$ is the one attached during that construction. By Property~\ref{G1} and Property~\ref{G2}, each $D_i$ is a disk.

By Property~\ref{G3}, we can assign a height to each $\widetilde D_i$ so that, in the fibered neighborhood $N(B)$, the embedded disks $\widetilde D_1,\widetilde D_2,\ldots,\widetilde D_k$ are pairwise disjoint and transverse to the $I$-fibers of $N(B)$.

Let $$F:= \bigcup_{i=1}^k \widetilde D_i$$
be our splitting surface in $N(B)$. Let $B'$ be the branched surface obtained from $B$ by splitting along $F$. By Property~\ref{G4}, $\partial F$ intersects each $I$-fiber in $\partial_v N(B)$ at most once. Hence, $B'$ is properly embedded in $S^3\setminus\operatorname{int}(\nu(L))$. 

\subsection{Properties of the branched surface \texorpdfstring{$B'$}{B'}}
In this subsection, we are going to prove that $B'$ is a laminar branched surface properly embedded in $S^3 \setminus \operatorname{int}(\nu(L))$. 

First, we show that the properties of $B$ established in Subsection~2.4 also hold for $B'$.
\begin{theorem}\label{cooriented-2}
    The branched surface $B'$ admits a co-orientation compatible with that of $B$.
\end{theorem}
\begin{proof}
    By Property~\ref{G1}, the splitting surface $F$ is transverse to every vertical line. The conclusion follows from the construction in \Cref{cooriented}.
\end{proof}

\begin{theorem}\label{disk-2}
    The sutured manifold 
\[
\bigl(
S^3 \setminus \operatorname{int}(\nu(L)\cup N(B')),
\ \overline{\partial \nu(L)\setminus N(B')}
\cup
\partial_v N(B')
\bigr)
\]
is homeomorphic to
\[
(D^2\times I,\ \partial D^2\times I).
\]
\end{theorem}
\begin{proof}
    The manifold \(S^3 \setminus \operatorname{int}(\nu(L)\cup N(B'))\) is obtained from \(S^3 \setminus \operatorname{int}(\nu(L)\cup N(B))\) by gluing $F\times I$ (that is, $k$ copies of $D^2\times I$) along $k$ rectangles in \(\partial_v N(B).\) The conclusion follows from \Cref{disk}.
\end{proof}

\begin{corollary}\label{transversal-2}
There exists a closed curve in $S^3\setminus  \operatorname{int}(\nu(L))$ positively transverse to $B'$, intersecting every branch sector of $B'$.
\end{corollary}
\begin{proof}
    Since \Cref{transversal} is a consequence of \Cref{cooriented} and \Cref{disk}, and since both \Cref{cooriented-2} and \Cref{disk-2} hold, we conclude that \Cref{transversal-2} holds as well.
\end{proof}

Since $L$ is non-split and non-trivial, $S^3\setminus  \operatorname{int}(\nu(L))$ is an irreducible, orientable $3$-manifold whose boundary is a union of incompressible tori. Now, we show that $\partial \nu(L)\setminus B'$ is a union of bigons.

\begin{theorem}\label{bigon}
    $\partial \nu(L)\setminus B'$ is a union of bigons.
\end{theorem}
\begin{proof}
By \Cref{disk-2},
\(
\overline{\partial \nu(L)\setminus  N(B')}
\cup
\partial_v N(B')
 \) is homeomorphic to an annulus $\partial D^2 \times I$. If $\partial \nu(L)\setminus B'$ is not a union of bigons, then $\overline{\partial \nu(L)\setminus N(B')}$ must be an annulus and $\partial_v N(B')$ must be empty. In this case, the branch locus of $B'$ is empty, and $B'$ is homeomorphic to a nontrivial compressing disk for $\partial \nu(L)$, which yields a contradiction.
\end{proof}

Now, we are ready to prove that $B'$ is a laminar branched surface.
    
\begin{theorem}\label{laminar}
The branched surface $B'$ is laminar.
\end{theorem}
\begin{proof}
We check the conditions for $B'$ to be a laminar branched surface one by one.

\begin{claim}
$B'$ contains no trivial bubble.
\end{claim}
\begin{proof}
    Recall that a trivial bubble is a $D^2 \times I$ region in $S^3 \setminus \operatorname{int}(\nu(L)\cup N(B'))$ such that each $I$-fiber of $N(B')$ intersects $\operatorname{int}(D^2) \times \partial I$ in at most one point. By \Cref{disk-2}, $S^3 \setminus\operatorname{int}(\nu(L)\cup N(B'))$ is homeomorphic to $D^2 \times I$, whose corresponding \(D^2 \times \partial I\) is \(\partial_h N(B')\). Therefore, almost all $I$-fibers of $N(B')$ intersect $\operatorname{int}(D^2) \times \partial I$ in two points. This contradicts the defining property of a trivial bubble.
\end{proof}

\begin{claim}
$\partial_h N(B')$ is incompressible and $\partial$-incompressible in \(S^3 \setminus \operatorname{int}(\nu(L)\cup N(B'))\).
\end{claim}
\begin{proof}
This is a direct corollary of \Cref{disk-2}.
\end{proof}
\begin{claim}
There is no monogon in \(S^3 \setminus \operatorname{int}(\nu(L)\cup N(B'))\).
\end{claim}
\begin{proof}
This is a direct corollary of \Cref{cooriented-2}.
\end{proof}
\begin{claim}
No component of $\partial_h N(B')$ is a sphere or a properly embedded disk in
$S^3 \setminus \operatorname{int}(\nu(L))$.
\end{claim}

\begin{proof}
By \Cref{disk-2}, the horizontal boundary $\partial_h N(B')$ is homeomorphic to the disjoint union of two disks, one on the positive side and the other on the negative side. In particular, no component of $\partial_h N(B')$ is a sphere.

By \Cref{bigon}, $\partial \nu(L)\setminus B'$ is a union of bigons. Near any vertex of a bigon, the two components of $\partial_h N(B')$ are not properly embedded in $S^3 \setminus \operatorname{int}(\nu(L))$.
\end{proof}
\begin{claim}
    \(S^3 \setminus \operatorname{int}(\nu(L)\cup N(B'))\) is irreducible and $\overline{\partial \nu(L)\setminus  N(B')}$ is incompressible in \(S^3 \setminus \operatorname{int}(\nu(L)\cup N(B'))\).
\end{claim}
\begin{proof}
    This is a direct corollary of \Cref{disk-2} and \Cref{bigon}.
\end{proof}
\begin{claim}
    $B'$ contains no Reeb branched surface, as defined in \cite{gabai1989essential}.
\end{claim}
\begin{proof}
    By \cite[Remark~1.3]{gabai1989essential}, $B'$ contains no Reeb branched surface provided that it is transversely recurrent, that is, there exists a closed curve efficiently transverse to $B'$ intersecting every branch sector of $B'$. Since any positively transverse curve is efficiently transverse, by \Cref{transversal-2}, $B'$ contains no Reeb branched surface.
\end{proof}
\begin{claim}
    $B'$ has no sink disk or half sink disk.
\end{claim}
\begin{proof}
By definition, a sink disk or a half sink disk is a branch sector of $B'$ such that
\begin{enumerate}[label=(\alph*)]
    \item it is homeomorphic to a disk, and
    \item its boundary consists only of arcs in $B' \cap \partial \nu(L)$ and arcs in the branch locus whose branch directions point inward.
\end{enumerate}

For each point $x' \in B'$, we are going to construct a curve $\gamma:[0,1]\to B'$ with $\gamma(0)=x'$ such that it meets the branch locus at least once, and whenever $\gamma$ meets the branch locus of $B'$, it follows the branch direction. Then, by definition, $x'$ cannot lie in the interior of a sink disk or a half sink disk.

By definition, if the splitting surface $F$ intersects an $I$-fiber $\{x\}\times I$ ($x\in B$) of $N(B)$ in $k$ points, then the point $x\in B$ splits into $k+1$ points $x' \in B'$. Thus, a point $x'\in B'$ is determined by a point $x\in B$ together with its relative height with respect to $F$.

We first consider the case where the corresponding point $x\in B$ lies in $S^1\times I\times [0,1)$. This is a region where no quotient map or splitting operation is performed. The set $B'\cap (S^1\times I\times [0,1))$ is a union of disks, each corresponding to an under-arc of $D$. By Property~\ref{G8}, there exists a curve in $B$ from $x$ to some point in $S^1\times I\times [1,2]$ that does not intersect the splitting surface $F$. This corresponds to a curve in $B'$ that does not intersect any branch locus. Therefore, we reduce the problem to the next case, where the corresponding point $x\in B$ lies in $S^1\times I\times [1,3]/\sim$.

Then, we assume that the corresponding point $x\in B$ lies in $S^1\times I\times [1,3]/\sim$. Let $(p,t) \in S^1 \times I\times [1,3]$ denote the point corresponding to $x'\in B'$. The curve $\gamma_0(s) = (p,\max(t,3s))$, where $s\in [0,1]$, together with an invariant relative height with respect to $F$, defines a curve from $x'$ to a point in one of the $k+1$ disks obtained by splitting the attached disk in the construction of $B$. Whenever this curve meets a branch locus of $B'$, it follows the branch direction. Therefore, we reduce the problem to the final case, where the corresponding point $x\in B$ lies in the attached disk in the construction of $B$.

Finally, we assume that the corresponding point $x\in B$ lies in the attached disk in the construction of $B$. Depending on the relative height with respect to $F$, the point $x'\in B'$ lies in one of the $k+1$ disks obtained by the splitting operation.

If $x'\in B'$ lies on the topmost disk, we choose a point $p\in \gamma_1\subset S^1\times I$ using Property~\ref{G5}. The curve $((p\times [0,3])/\sim)\cap B$ in $B$ lifts to a curve in $B'$, and it intersects the branch locus of $B'$ exactly once at $p\times \{1\}$, where the branch direction points outward from the topmost disk. Therefore, there exists a curve $\gamma$ in $B'$ starting at $x'$ such that it meets at least one branch locus, and whenever $\gamma$ meets a branch locus of $B'$, it follows the branch direction.

If $x'\in B'$ lies on the bottommost disk, then we similarly choose a point $p\in \gamma_k\subset S^1\times I$ using Property~\ref{G6}. The same conclusion follows.

If $x'\in B'$ lies between $\widetilde{D}_i$ and $\widetilde{D}_{i+1}$ for some $1\le i\le k-1$, we choose points $p_1\in \gamma_{i}\subset S^1\times I$ and $p_2\in \gamma_{i+1}\subset S^1\times I$ using Property~\ref{G7}. Let $p\in \{p_1,p_2\}$ be the point lying on an under-arc of $\tau$. The same conclusion follows.

Therefore, $B'$ has no sink disk or half sink disk.
\end{proof}

By Definition~\ref{def-laminar}, $B'$ is a laminar branched surface.
\end{proof}
\section{Curves for the splitting surface}\label{sec:curve}
In this section, we are going to construct simple closed curves \(\gamma_1, \gamma_2, \ldots, \gamma_k \subset S^1 \times I\) satisfying Properties~\ref{G1}--\ref{G8}.

\subsection{The brick diagram}
Consider the brick diagram associated to the positive braid diagram $D\subset S^1\times I$, drawn horizontally. It is a diagram in $S^1\times I$ obtained by replacing each crossing with a vertical segment connecting two adjacent rows. 

A row in the brick diagram is a horizontal circle in $S^1\times I$. The $i$-th row (counted from top to bottom) corresponds to the set of points $p$ in the braid diagram $D$ such that there are $i-1$ strands directly above $p$. Each bigon complementary region of $\tau$ is represented by an innermost rectangle in the brick diagram, which we call a \emph{brick}.

Each crossing in $D$ corresponds to a vertical segment in the brick diagram, which in turn determines two vertices. We call the top endpoint of the segment an \emph{up-vertex} and the bottom endpoint a \emph{down-vertex}. For each vertex $p$ of the brick diagram, let $\opp(p)$ denote the other vertex corresponding to the same crossing. Let $[p,\opp(p)]$ denote the vertical segment joining $p$ and $\opp(p)$.

For two points $p_1, p_2$ on the same row in the brick diagram, we consider the oriented row from left to right. Let $[p_1,p_2]$ denote the horizontal segment obtained by traveling from $p_1$ to $p_2$ in this direction; if $p_1=p_2$, we traverse the entire row once. Let $d_{\mathrm{up}}(p_1,p_2)$ denote the number of up-vertices on $[p_1,p_2]$, excluding the endpoints. Similarly, let $d_{\mathrm{down}}(p_1,p_2)$ denote the number of down-vertices on $[p_1,p_2]$, again excluding the endpoints. These functions are analogous to those defined in \cite[Definition~3.7]{krishna2025taut}.

For each point $p$ on a row in the brick diagram, if $d_{\mathrm{up}}(p,p)\ge 1$, let $L_{\mathrm{up}}(p)$ denote the nearest up-vertex to the left of $p$ on its row, and let $R_{\mathrm{up}}(p)$ denote the nearest up-vertex to the right of $p$ on its row. Similarly, if $d_{\mathrm{down}}(p,p)\ge 1$, let $L_{\mathrm{down}}(p)$ denote the nearest down-vertex to the left of $p$ on its row, and let $R_{\mathrm{down}}(p)$ denote the nearest down-vertex to the right of $p$ on its row.

The brick diagram determines the braid word as follows. Scan the brick diagram from left to right. Whenever a vertical segment connecting the $i$-th row to the $(i+1)$-st row is encountered, we write down the generator $\sigma_i$. Listing these generators in the order they appear produces the braid word associated to the diagram.

\subsection{Construction of the curves}
We describe a linear-complexity greedy algorithm that produces simple closed curves $\gamma_1,\ldots,\gamma_k$ satisfying Properties~\ref{G1}--\ref{G8}. The precise construction is given in Algorithm~\ref{alg:brick-curve}. The local moves used in the construction are summarized in Algorithm~\ref{alg:local-procedures}. Figure~\ref{fig:brick} illustrates the construction on a sample brick diagram.

\begin{figure}[!htbp]
  \centering
  \includegraphics[width=8cm]{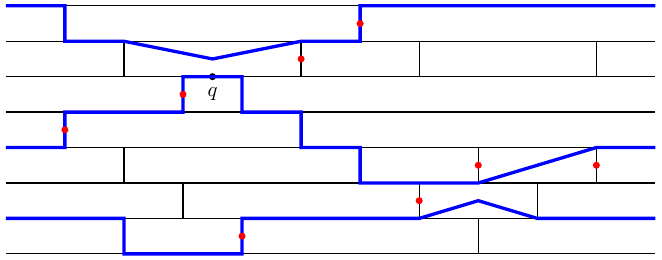}
  \caption{Blue curves represent the output; red dots represent Type~X crossings.}

  \label{fig:brick}
\end{figure}

\begin{algorithm}[!htbp]
\caption{Construction of the curves}
\label{alg:brick-curve}
\begin{algorithmic}[1]
\Require A positive braid diagram $D\subset S^1\times I$ and its associated brick diagram
\Ensure Simple closed curves satisfying Properties~\ref{G1}--\ref{G8}

\State Choose a non-vertex point $q$ in the $i$-th row lying on a component of $L$ which is not the unknot, with $i$ minimal.
\State $A\gets q$, $B\gets q$, $\gamma\gets\{q\}$, $\Gamma\gets\varnothing$
\While{$d_{\mathrm{up}}(B,A)=2$}
    \State \Call{StepBelow}{$A,B,\gamma$}
\EndWhile
\State \Call{CrossBrickStep}{$A,B,\gamma$} \label{line:CrossBrickStep}
\While{$A$ is not on the last row}
    \If{$d_{\mathrm{up}}(B,A)\ge2$}
        \State \Call{StepBelow}{$A,B,\gamma$}
    \Else
        \State \Call{CloseAndStore}{$\gamma,\Gamma,[B,A]$} \label{line:close-1}
        \State \Call{RestartBelow}{$A,B,\gamma$} \label{line:restart-below}
    \EndIf
\EndWhile
\State \Call{CloseAndStore}{$\gamma,\Gamma,[B,A]$}\label{line:close-2}
\If{$q$ is not on the first row}
    \State $A\gets q$, $B\gets q$
    \State \Call{RestartAbove}{$A,B,\gamma$}\label{line:restart-above1}
    \While{$A$ is not on the first row}
        \If{$d_{\mathrm{up}}(A,B)\ge2$}
            \State \Call{StepAbove}{$A,B,\gamma$}
        \Else
            \State \Call{CloseAndStore}{$\gamma,\Gamma,[A,B]$}\label{line:close-3}
            \State \Call{RestartAbove}{$A,B,\gamma$}\label{line:restart-above2}
        \EndIf
    \EndWhile
    \State \Call{CloseAndStore}{$\gamma,\Gamma,[A,B]$}\label{line:close-4}
\EndIf

\State \Return $\Gamma$
\end{algorithmic}

\end{algorithm}

\begin{algorithm}[!htbp]
\caption{Local procedures used in the Algorithm~\ref{alg:brick-curve}}
\label{alg:local-procedures}
\begin{algorithmic}[1]

\Procedure{CrossBrickStep}{$A,B,\gamma$}
    \State $A' \gets L_{\mathrm{up}}(A)$,\quad $B' \gets R_{\mathrm{up}}(B)$
    \State $A'' \gets L_{\mathrm{up}}(A')$
    \State Extend $\gamma$ along $[A',A]$, $[B,B']$, and $[B',\opp(B')]$
    \State Extend $\gamma$ inside the brick below $[A'',A']$ along an arc from $\opp(A'')$ to $A'$
    \State $(A,B) \gets (\opp(A''),\,\opp(B'))$
\EndProcedure

\Procedure{StepBelow}{$A,B,\gamma$}
    \State $A' \gets L_{\mathrm{up}}(A)$,\quad $B' \gets R_{\mathrm{up}}(B)$
    \State Extend $\gamma$ along $[A',A]$, $[B,B']$, $[A',\opp(A')]$, and $[B',\opp(B')]$
    \State $(A,B) \gets (\opp(A'),\,\opp(B'))$
\EndProcedure

\Procedure{StepAbove}{$A,B,\gamma$}
    \State $A' \gets R_{\mathrm{down}}(A)$,\quad $B' \gets L_{\mathrm{down}}(B)$
    \State Extend $\gamma$ along $[B',B]$, $[A,A']$, $[A',\opp(A')]$, and $[B',\opp(B')]$
    \State $(A,B) \gets (\opp(A'),\,\opp(B'))$
\EndProcedure

\Procedure{CloseAndStore}{$\gamma,\Gamma,\mathrm{Seg}$}
    \State Extend $\gamma$ along $\mathrm{Seg}$ to close $\gamma$
    \State $\Gamma \gets \Gamma \cup \{\gamma\}$
\EndProcedure

\Procedure{RestartBelow}{$A,B,\gamma$}
    \State $R \gets$ the brick directly below $A$
    \State $A \gets$ the lower-left vertex of $R$
    \State $B \gets$ the lower-right vertex of $R$
    \State $\gamma \gets$ an arc inside $R$ from $A$ to $B$
\EndProcedure

\Procedure{RestartAbove}{$A,B,\gamma$}
    \State $R \gets$ the brick directly above $A$
    \State $A \gets$ the upper-right vertex of $R$
    \State $B \gets$ the upper-left vertex of $R$
    \State $\gamma \gets$ an arc inside $R$ from $B$ to $A$
\EndProcedure

\end{algorithmic}
\end{algorithm}

During the construction, we maintain two points $A$ and $B$ on a common row, along with a path $\gamma$ connecting $A$ and $B$ in $S^1\times I$. The path $\gamma$ has already been constructed above or below this row, and the next modification of $\gamma$ is determined by the number of up-vertices (or down-vertices) inside $[B,A]$ (or $[A,B]$).

The construction proceeds in three phases: (i) locating the middle brick to cross, (ii) downward expansion to the last row, and (iii) upward expansion to the first row.

\subsubsection{Phase (i): locating the middle brick to cross.}
We start from a non-vertex point $q$ on a component of $L$ which is not the unknot, chosen on the topmost row met by this component. We set $A=B=q$ and initialize $\gamma$ to be the constant path at $q$. As long as $d_{\mathrm{up}}(B,A)=2$, the procedure \textsc{StepBelow} extends $\gamma$ by moving $A$ to the nearest up-vertex on the left and $B$ to the nearest up-vertex on the right, and then descending along the corresponding vertical segments to the paired down-vertices.

Once $d_{\mathrm{up}}(B,A)\neq 2$, as we will prove later in Subsection~\ref{sec:simple}, we must have $d_{\mathrm{up}}(B,A)\ge 3$. The procedure \textsc{CrossBrickStep} extends $\gamma$ by moving $A$ to the nearest up-vertex on the left and then passing through the interior of the brick toward its bottom-left corner, while moving $B$ to the nearest up-vertex on the right and descending along the corresponding vertical segment.

After this step, both $A$ and $B$ are down-vertices, and the construction enters the downward phase.

\subsubsection{Phase (ii): downward expansion to the last row.}
If $d_{\mathrm{up}}(B,A)\ge 2$, we extend $\gamma$ further downward by applying \textsc{StepBelow} and continue the downward expansion.

If instead $d_{\mathrm{up}}(B,A)\le 1$, we close the current path $\gamma$ by adding the segment $[B,A]$ and store it in the output collection $\Gamma$. If this is not the last row, we start a new path inside the brick directly below $A$ (procedure \textsc{RestartBelow}), setting $(A,B)$ to the lower-left and lower-right vertices of that brick and letting $\gamma$ be an interior arc from $A$ to $B$. Otherwise, the downward expansion terminates.

\subsubsection{Phase (iii): upward expansion to the first row.}
This phase is symmetric to the downward phase. If $q$ is not on the first row, we restart a path $\gamma$ inside the brick directly above $q$ (procedure \textsc{RestartAbove}) and then iterate upward.

\subsection{The output curves are simple}\label{sec:simple}

In this subsection, we prove that Algorithm~\ref{alg:brick-curve} produces simple closed curves. During construction, every step is performed on a single row, proceeding first downward from the initial point $q$ and then upward from $q$. Therefore, it suffices to show that the construction on each individual row produces no self-intersection.

There are two statements to verify:
\begin{itemize}
    \item The two vertical segments of each brick are distinct. Hence, after \textsc{RestartBelow} (line~\ref{line:restart-below}) and \textsc{RestartAbove} (lines~\ref{line:restart-above1} and~\ref{line:restart-above2}), the vertices $A$ and $B$ are distinct.
    \item Before \textsc{CrossBrickStep} (line~\ref{line:CrossBrickStep}), $d_{\mathrm{up}}(B,A)\ge 3$. Therefore, the segment added during \textsc{CrossBrickStep} also cannot create a self-intersection.
\end{itemize}

We prove these two statements using braid group relations by reading a braid word from the brick diagram. For the first statement, we establish the following lemma. This proof follows the same argument as \cite[Lemma~4.1]{krishna2025taut} and can be visualized by the isotopy shown in Figure~\ref{fig:strand-1}.

\begin{figure}[!htbp]
\centering

\newcommand{\casefig}[1]{%
\includegraphics[width=\linewidth]{#1}%
}

\begin{tabular}{c c c}
\begin{minipage}{0.44\textwidth}\centering
\casefig{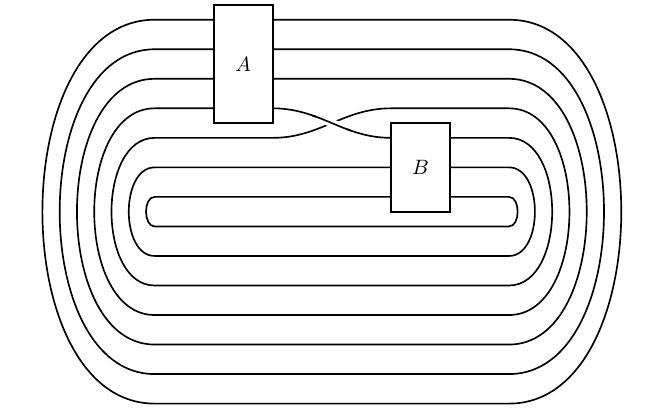}
\end{minipage}
&
$\cong$
&
\begin{minipage}{0.44\textwidth}\centering
\casefig{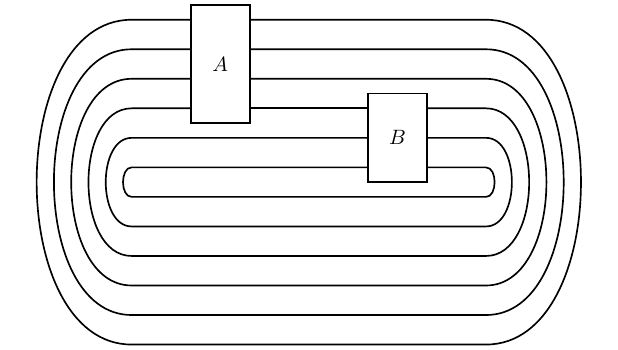}
\end{minipage}
\end{tabular}

\caption{Illustration of the isotopy used in Lemma~\ref{strand-1}.}
\label{fig:strand-1}
\end{figure}

\begin{lemma}\label{strand-1}
Let $\beta$ be a positive braid word on $m$ strands. Suppose there exists an index $i$ such that $\sigma_i$ appears only once in $\beta$. Then there exists a positive braid word $\beta'$ on $m-1$ strands such that the closures $\widehat\beta$ and $\widehat{\beta'}$ represent the same link in $S^3$.
\end{lemma}

\begin{proof}
For each $k$, let $\pi_k$ denote the braid \(\pi_k := \sigma_{m-1}\sigma_{m-2}\cdots \sigma_{k}.\)
By the braid relations, we have
\[
\pi_k \sigma_j\pi_k^{-1}=
\begin{cases}
    \sigma_{j-1}, & \text{if } j>k,\\
    \sigma_{j}, & \text{if } j <k-1,
\end{cases}
\]
and
\[
\pi_k \pi_{k+1}^{-1}=\pi_{k}^{-1}\sigma_{m-1}\pi_{k}.
\]

By the braid relations, a conjugate of $\beta$ can be written as the product $W_0 \sigma_i W_1$, where $W_0$ is a positive word  in the generators $\sigma_1 ,\sigma_2,\ldots, \sigma_{i-1}$ and $W_1$ is a positive word in the generators $\sigma_{i+1},\sigma_{i+2},\ldots, \sigma_{m-1}$. Then we have
\begin{align*}
    \pi_{i} \beta \pi_{i}^{-1}
=&\pi_{i} W_0 \sigma_i W_1   \pi_{i}^{-1}\\
=&(\pi_i \pi_{i+1}^{-1})(\pi_{i+1} W_0\pi_{i+1}^{-1}) (\pi_{i}W_1 \pi_{i}^{-1})\\
=&(\pi_{i}^{-1}\sigma_{m-1}\pi_{i})  W_0 (\pi_{i}W_1 \pi_{i}^{-1}).
\end{align*}

The element $\pi_{i}W_1 \pi_{i}^{-1}$ can be written as a positive word in the generators $\sigma_{i},\sigma_{i+1},\ldots, \sigma_{m-2}$.
After canceling a $\sigma_{m-1}^{-1}\sigma_{m-1}$ pair, the word $\pi_i^{-1}\sigma_{m-1}\pi_i$ contains exactly one occurrence of $\sigma_{m-1}$ and no $\sigma_{m-1}^{-1}$.

Hence, $\pi_i \beta \pi_i^{-1}$ admits a destabilization along the last strand, which removes this $\sigma_{m-1}$ and produces a braid word $\beta'$ in the generators $\sigma_1,\sigma_2,\ldots,\sigma_{m-2}$.
The closures $\widehat{\beta}$ and $\widehat{\beta'}$ represent the same link in $S^3$.
\end{proof}

For the second statement, observe that each time \textsc{StepBelow} is performed, the vertices $A$ and $B$ move down by one row. Since there are only finitely many rows, once the first instance with $d_{\mathrm{up}}(B,A)\neq 2$ occurs, the algorithm enters line~\ref{line:CrossBrickStep}. To prove $d_{\mathrm{up}}(B,A)\ge 3$, it therefore suffices to rule out the cases $d_{\mathrm{up}}(B,A)=0$ and $d_{\mathrm{up}}(B,A)=1$.

Suppose that the initial point $q$ is on the $i$-th row, and \textsc{StepBelow} is performed $j$ times before entering line~\ref{line:CrossBrickStep}. By conjugation, we may assume that the braid word $\beta$ on $m$ strands is 
\[
\beta =
\sigma_{i+j-1}\, w_{-j+1}\cdots
\sigma_{i+1}\, w_{-1}\,
\sigma_i\, w_0\, \sigma_i\,
w_1\, \sigma_{i+1}\, w_2\cdots
w_{j-1}\, \sigma_{i+j-1}\, w_j ,
\]
where $w_{-j+1},\ldots,w_{-1},w_0,w_1,\ldots,w_j$ are positive words. 

By the condition that $d_{\mathrm{up}}(B,A)=2$ before each execution of \textsc{StepBelow}, prior to entering line~\ref{line:CrossBrickStep}, for $k=0,1,\ldots,j-1$ and for every $l$ satisfying $|l|\ge k$, the word $w_l$ contains no $\sigma_{i+k}$. By the condition that $i$ is minimal in the choice of the initial point $q$, the word $w_0$ contains no $\sigma_{i-1}$. 

For \(l=1,2,\dots, j-1\), the words \(w_{-l}\) and \(w_l\) are positive words in the generators
\[
\sigma_1,\sigma_2,\ldots, \sigma_{i-1}, \ \sigma_{i+l+1},\sigma_{i+l+2},\ldots, \sigma_{m-1}.
\]
By the braid relations, each of \(w_{-l}\) and \(w_l\) can be written as a product of a positive word (called the \emph{upper part}) in the generators \(\sigma_1,\sigma_2,\ldots, \sigma_{i-1}\), and a positive word (called the \emph{lower part}) in the generators \(\sigma_{i+l+1},\sigma_{i+l+2},\ldots, \sigma_{m-1}\).

Since \(w_0\) does not contain \(\sigma_{i-1}\) or \(\sigma_i\), it can also be written as a product of a positive word (called the \emph{upper part}) in the generators \(\sigma_1,\sigma_2,\ldots, \sigma_{i-2}\), and a positive word (called the \emph{lower part}) in the generators \(\sigma_{i+1},\sigma_{i+2},\ldots, \sigma_{m-1}\).

By the braid relations, the upper part of \(w_{l}\) commutes with \(\sigma_{i+|l|},\ldots ,\sigma_{i+j-1}\), the lower part of \(w_l\) commutes with \(\sigma_i,\ldots ,\sigma_{i+|l|-1}\), and every upper part commutes with every lower part.

Hence, we can push each lower part into the central word \(w_0\), and push each upper part into the final word \(w_j\), possibly after conjugation. Therefore, a conjugate of \(\beta\) can be written as
\[
\beta'= \sigma_{i+j-1}\cdots \sigma_{i+1}\sigma_i \, W_0 \,\sigma_i \sigma_{i+1}\cdots\sigma_{i+j-1}\, W_1,
\]
where \(W_0\) is a positive word in the generators \(\sigma_{i+1},\sigma_{i+2},\ldots, \sigma_{m-1}\), and \(W_1\) is a positive word in the generators \(\sigma_1,\sigma_2,\ldots,\sigma_{i-1},\sigma_{i+j},\sigma_{i+j+1},\ldots,\sigma_{m-1}\).

The value of $d_{\mathrm{up}}(B,A)$ upon entering line~\ref{line:CrossBrickStep} is equal to the number of occurrences of the letter $\sigma_{i+j}$ in $w_j$, which is also equal to the number of occurrences of $\sigma_{i+j}$ in $W_1$.

If $d_{\mathrm{up}}(B,A)=0$, then the braid word $\beta'$ maps the $(i+j)$-th strand to itself. Consequently, the closure $L$ contains an unknot component passing through the initial point $q$, contradicting the choice of $q$. 

If $d_{\mathrm{up}}(B,A)=1$, we consider two cases. If $j=0$, we apply Lemma~\ref{strand-1} to construct a positive braid on one fewer strand representing the same link. Otherwise, we apply the following lemma to construct a positive braid on one fewer strand representing the same link. In either case, this contradicts the choice of the braid diagram. The proof of Lemma~\ref{strand-2} can be visualized by the isotopy shown in Figure~\ref{fig:strand-2}.

\begin{figure}[!htbp]
\centering

\newcommand{\casefig}[1]{%
\includegraphics[width=\linewidth]{#1}%
}

\begin{tabular}{c c c}
\begin{minipage}{0.44\textwidth}\centering
\casefig{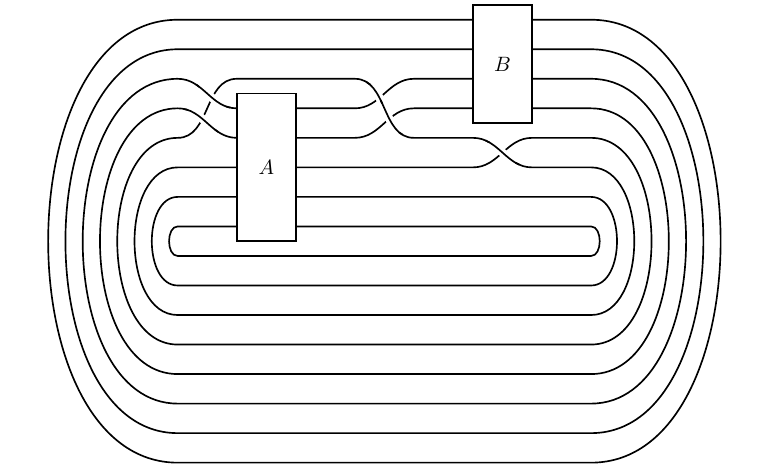}
\end{minipage}
&
$\cong$
&
\begin{minipage}{0.44\textwidth}\centering
\casefig{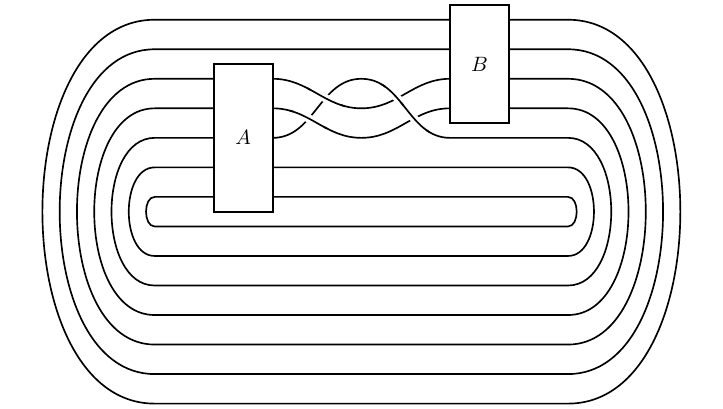}
\end{minipage}
\end{tabular}

\caption{Illustration of the isotopy used in Lemma~\ref{strand-2}.}
\label{fig:strand-2}
\end{figure}

\begin{lemma}\label{strand-2}
Let
\[
\beta= \sigma_{i+j-1}\cdots \sigma_{i+1}\sigma_i \, W_0 \,\sigma_i \sigma_{i+1}\cdots\sigma_{i+j-1}\, W_1
\]
be a positive braid word on $m$ strands. Suppose that $W_0$ contains neither $\sigma_{i-1}$ nor $\sigma_i$, and $W_1$ contains no $\sigma_{i+j-1}$ and exactly one occurrence of $\sigma_{i+j}$. Then there exists a positive braid word $\beta'$ on $m-1$ strands such that the closures $\widehat{\beta}$ and $\widehat{\beta'}$ represent the same link in $S^3$.
\end{lemma}

\begin{proof}
For each $k$, let $\pi_k$ denote the braid \(\pi_k := \sigma_{m-1}\sigma_{m-2}\cdots \sigma_{k}.\) By the braid relations, we have
\[
\pi_k \sigma_j\pi_k^{-1}=
\begin{cases}
    \sigma_{j-1}, & \text{if } j>k,\\
    \sigma_{j}, & \text{if } j <k-1,
\end{cases}
\]
and
\[
\pi_k \pi_{k+1}^{-1}=\pi_{k}^{-1}\sigma_{m-1}\pi_{k}.
\]

By the braid relations, the element $W_0$ can be written as the product $V_0V_1$, where $V_0$ is a positive word in the generators $\sigma_1 ,\sigma_2,\ldots, \sigma_{i-2}$ and $V_1$ is a positive word in the generators $\sigma_{i+1},\sigma_{i+2},\ldots, \sigma_{m-1}$. Similarly, the element $W_1$ can be written as the product $V_2V_3\sigma_{i+j} V_4  V_5$, where $V_2$ and $V_4$ are positive words in the generators $\sigma_1 ,\sigma_2,\ldots, \sigma_{i+j-2}$, and $V_3$ and $V_5$ are positive words in the generators $\sigma_{i+j+1},\sigma_{i+j+2},\ldots, \sigma_{m-1}$. Then we have
\begin{align*}
    &\pi_{i+j}\beta \pi_{i+j}^{-1} \\
    =&\pi_{i}V_0V_1 \sigma_{i}\sigma_{i+1}\cdots \sigma_{i+j-1} V_2V_3\sigma_{i+j} V_4V_5\pi_{i+j}^{-1}\\
    = &\pi_{i}V_0V_1V_3 \sigma_{i}\sigma_{i+1}\cdots \sigma_{i+j-1} \sigma_{i+j} V_2V_4V_5\pi_{i+j}^{-1}\\
    =&(\pi_i V_0 \pi_i^{-1})(\pi_i V_1V_3 \pi_i^{-1})(\pi_i \sigma_{i}\sigma_{i+1}\cdots \sigma_{i+j}\pi_{i+j}^{-1})(\pi_{i+j}V_2V_4\pi_{i+j}^{-1})(\pi_{i+j}V_5\pi_{i+j}^{-1})\\
    =&V_0(\pi_i V_1V_3 \pi_i^{-1})(\pi_i \sigma_{i}\sigma_{i+1}\cdots \sigma_{i+j-1}\pi_{i+j+1}^{-1})V_2 V_4(\pi_{i+j}V_5\pi_{i+j}^{-1})\\
    =&V_0(\pi_i V_1V_3 \pi_i^{-1})(\pi_{i+j}\pi_{i+j+1}^{-1})( \sigma_{i+j-1} \cdots \sigma_{i}\sigma_{i}\cdots \sigma_{i+j-1})V_2 V_4(\pi_{i+j}V_5\pi_{i+j}^{-1})\\
    =&V_0(\pi_i V_1V_3 \pi_i^{-1})(\pi_{i+j}^{-1}\sigma_{m-1}\pi_{i+j})( \sigma_{i+j-1} \cdots \sigma_{i}\sigma_{i}\cdots \sigma_{i+j-1})V_2 V_4(\pi_{i+j}V_5\pi_{i+j}^{-1}).
\end{align*}

The element $\pi_i V_1V_3 \pi_i^{-1}$ can be written as a positive word in the generators $\sigma_{i},\sigma_{i+1},\ldots, \sigma_{m-2}$, and the element $\pi_{i+j}V_5\pi_{i+j}^{-1}$ can be written as a positive word in the generators $\sigma_{i+j},\sigma_{i+j+1},\ldots, \sigma_{m-2}$.
After canceling a $\sigma_{m-1}^{-1}\sigma_{m-1}$ pair, the word $\pi_{i+j}^{-1}\sigma_{m-1}\pi_{i+j}$ contains exactly one occurrence of $\sigma_{m-1}$ and no $\sigma_{m-1}^{-1}$.

Hence, $\pi_{i+j} \beta \pi_{i+j}^{-1}$ admits a destabilization along the last strand, which removes this $\sigma_{m-1}$ and produces a braid word $\beta'$ in the generators $\sigma_1,\sigma_2,\ldots,\sigma_{m-2}$.
The closures $\widehat{\beta}$ and $\widehat{\beta'}$ represent the same link in $S^3$.
\end{proof}

Therefore, the output curves of Algorithm~\ref{alg:brick-curve} are simple.

\subsection{Verification of Properties \texorpdfstring{\ref{G1}--\ref{G8}}{(G1)--(G8)}}\label{sec:verify-G}

Let \(\gamma_1,\gamma_2,\ldots,\gamma_k\) be the simple closed curves produced by Algorithm~\ref{alg:brick-curve}. Assume that \(\gamma_s\) is the first curve constructed (so the initial point \(q\) lies on \(\gamma_s\)). Moreover, assume that \(\gamma_{s+1},\ldots,\gamma_k\) are the curves constructed during the remainder of the downward expansion phase, and that \(\gamma_{s-1},\ldots,\gamma_1\) are the curves constructed during the upward expansion phase. We will prove that \(\gamma_1,\gamma_2,\ldots,\gamma_k\) satisfy Properties~\ref{G1}--\ref{G8}.

\begin{theorem}\label{thm:G1-G8}
The simple closed curves \(\gamma_1,\gamma_2,\ldots,\gamma_k\) satisfy Properties~\ref{G1}--\ref{G8}.
\end{theorem}
\begin{proof}
We verify the properties one by one.

\medskip
\noindent\textbf{Property~\ref{G1}.}
During the construction of $\gamma_s,\gamma_{s+1},\ldots,\gamma_k$, the path is always extended to the left of the left endpoint $A$ and to the right of the right endpoint $B$. During the construction of $\gamma_{s-1},\gamma_{s-2},\ldots,\gamma_1$, the path is always extended to the right of the right endpoint $A$ and to the left of the left endpoint $B$. Therefore, the output curves are transverse to every vertical line.

\medskip
\noindent\textbf{Property~\ref{G2}.}
The curve $\gamma_s$ contains an interior arc connecting two corners of a brick constructed in the \textsc{CrossBrickStep} procedure at line~\ref{line:CrossBrickStep}. Each curve in $\gamma_{s+1},\ldots,\gamma_k$ contains an interior arc connecting two corners of a brick constructed in the \textsc{RestartBelow} procedure at line~\ref{line:restart-below}. Each curve in $\gamma_{s-1},\ldots,\gamma_1$ contains an interior arc connecting two corners of a brick constructed in the \textsc{RestartAbove} procedure at line~\ref{line:restart-above1} or line~\ref{line:restart-above2}. The remaining parts follow the edges of the bricks.

\medskip
\noindent\textbf{Property~\ref{G3}.}
Suppose the initial point $q$ lies on row $r_s$. Each time the \textsc{CloseAndStore} procedure at line~\ref{line:close-1} or line~\ref{line:close-2} is performed, the endpoints $A$ and $B$ lie on the rows $r_{s+1},\ldots,r_{k+1}$, respectively. Afterwards, each time the \textsc{CloseAndStore} procedure at line~\ref{line:close-3} or line~\ref{line:close-4} is performed, the endpoints $A$ and $B$ lie on the rows $r_{s-1},\ldots,r_1$, respectively. Consequently, each $\gamma_i$ lies between row $r_i$ and row $r_{i+1}$. Since $r_1<r_2<\cdots< r_{k+1}$, Property~\ref{G3} holds.

\medskip
\noindent\textbf{Property~\ref{G4}.}
Since each $\gamma_i$ lies between row $r_i$ and row $r_{i+1}$, the associated bricks are also distinct.

\medskip
\noindent\textbf{Property~\ref{G5}.}
If the initial point $q$ is not on the first row, a point immediately to the right of $A$ when performing the \textsc{CloseAndStore} procedure at line~\ref{line:close-4} serves as the point $p$ in Property~\ref{G5}. Otherwise, the point immediately to the right of $L_{\mathrm{down}}(q)$ serves as the point $p$ in Property~\ref{G5}.

\medskip
\noindent\textbf{Property~\ref{G6}.}
A point immediately to the left of $A$ when performing the \textsc{CloseAndStore} procedure at line~\ref{line:close-2} serves as the point $p$ in Property~\ref{G6}.

\medskip
\noindent\textbf{Property~\ref{G7}.}
For each $s\le i\le k-1$, the point immediately to the left of $A$ when performing the \textsc{CloseAndStore} procedure at line~\ref{line:close-1} on row $r_{i+1}$ serves as the point $p_1\in \gamma_i$ in Property~\ref{G7}. 
The \textsc{RestartBelow} procedure at line~\ref{line:restart-below} initiates the next curve in the brick directly below $A$. 
Thus, directly below $p_1$ there is a point serving as $p_2\in \gamma_{i+1}$ in Property~\ref{G7}.

For $i=s-1$, the initial point $q$ serves as the point $p_2\in \gamma_{i+1}$ in Property~\ref{G7}. By the choice of $q$, we may assume that $q$ lies on an under-type arc of $\tau$. The \textsc{RestartAbove} procedure at line~\ref{line:restart-above1} initiates the next curve in the brick directly above $p_2$. Thus, directly above $p_2$ there is a point serving as $p_1\in \gamma_{i}$ in Property~\ref{G7}.

For each $1\le i\le s-2$, the point immediately to the right of $A$ when performing the \textsc{CloseAndStore} procedure at line~\ref{line:close-3} on row $r_{i+1}$ serves as the point $p_2\in \gamma_{i+1}$ in Property~\ref{G7}. Thus, directly above $p_2$ there is a point serving as $p_1\in \gamma_{i}$ in Property~\ref{G7}.

\medskip
\noindent\textbf{Property~\ref{G8}.}
An under-arc of $D$ corresponds to a collection of under-type arcs of $\tau$. These under-type arcs can be represented in the brick diagram as follows.

Let
\[
U_0,U_1,D_1,U_2,D_2,\ldots,U_t,D_t,D_{t+1}
\]
be vertices of the brick diagram such that
\begin{itemize}
    \item each $U_i$ is an up-vertex;
    \item each $D_i$ is a down-vertex;
    \item each $[U_i,D_i]$ represents a vertical segment;
    \item $[U_1,U_0]$, $[D_{t+1},D_t]$, and each $[U_{i+1},D_i]$ represent horizontal segments with no interior vertices. The left half of $[U_1,U_0]$, the right half of $[D_{t+1},D_t]$, and each segment $[U_{i+1},D_i]$ correspond to the $t+1$ under-type arcs of $\tau$ arising from a single under-arc.

\end{itemize}
If there exists an interior point in $[U_1,U_0]$, $[D_{t+1},D_t]$, or some $[U_{i+1},D_i]$ that is not on any output curve, then Property~\ref{G8} holds. We consider three cases.

First, suppose that the row containing $[U_1,U_0]$ lies below the initial point $q$. In this case, none of the local procedures extends a path along $[U_1,U_0]$. Thus, the interior of $[U_1,U_0]$ does not intersect any output curve.

Second, suppose that the row containing $[D_{t+1},D_t]$ lies above the initial point $q$. Similarly, none of the local procedures extends a path along $[D_{t+1},D_t]$. Thus, the interior of $[D_{t+1},D_t]$ does not intersect any output curve.

Finally, suppose that the row containing $[U_1,U_0]$ is not below the initial point $q$, and the row containing $[D_{t+1},D_t]$ is not above $q$. Then there exists a segment $I$ among $[U_1,U_0]$, $[D_{t+1},D_t]$, or some $[U_{i+1},D_i]$ that lies on the same row as $q$.

By Algorithm~\ref{alg:brick-curve}, the initial row intersects the output curves in the interval $[L_{\mathrm{up}}(q),R_{\mathrm{up}}(q)]$. By the choice of $q$, there are no interior vertices in $[L_{\mathrm{up}}(q),R_{\mathrm{up}}(q)]$. If $I$ intersects some output curve, then it must coincide with $[L_{\mathrm{up}}(q),R_{\mathrm{up}}(q)]$. Since both endpoints are up-vertices, we must have $I=[U_1,U_0]$.

Every time we update $A$ before \textsc{CrossBrickStep} (line~\ref{line:CrossBrickStep}), the updated value is some $D_i$. During \textsc{CrossBrickStep}, the point immediately to the left of $\opp(A')$ is an interior point of $[D_{t+1},D_t]$ or of some $[U_{i+1},D_i]$, and it does not lie on any output curve.

This completes the proof.
\end{proof}
\section{Slopes realized by the boundary train track}\label{sec:5}

Following \cite{li2003boundary}, we study the slopes realized by the boundary train track of the branched surface $B'$. Here, a rational slope $s\in \mathbb{Q}\cup \{\infty\}$ is said to be \emph{realized} by a train track $\tau$ if $\tau$ fully carries a union of simple closed curves with slope $s$.

\subsection{From carried slopes to fully carried slopes}

The purpose of this subsection is to pass from slopes realized by a sub-train track to slopes realized by the entire train track. In particular, if our boundary train track carries curves of two distinct rational slopes $s_1$ and $s_2$, then it realizes every rational slope in an open interval between $s_1$ and $s_2$.

\begin{theorem}\label{subtrack-realize}
Let $\tau$ be an oriented train track on a torus $T$ such that $T\setminus \tau$ is a union of bigons. Let $\tau'\subset\tau$ be a sub-train track realizing two distinct rational slopes. Then every rational slope realized by $\tau'$ is also realized by $\tau$.
\end{theorem}

\begin{proof}
Let $s_1$ and $s_2$ be distinct rational slopes realized by $\tau'$. We prove that $s_1$ is also realized by $\tau$. 
Let $\gamma_1$ (resp. $\gamma_2$) be a closed curve of slope $s_1$ (resp. $s_2$) carried by $\tau'$. Let $\rho_1$ and $\rho_2$ be the corresponding train routes. Since $s_1\neq s_2$, they intersect; fix an intersection point $p_0\in \rho_1\cap\rho_2$.

Take any point $p\in \tau$. Starting at $p$, follow the orientation of $\tau$. Since $\tau$ is compact, the path must eventually meet itself. The first self-intersection produces a closed train route $\gamma$.

Each complementary region of $\tau$ is a bigon, hence has cusped Euler index $0$. Therefore, any union of complementary regions also has index $0$. Since a disk has cusped Euler index $1$, $\tau$ cannot carry a null-homotopic closed curve. Hence $\gamma$ represents a slope in $T$.

Because $s_1\neq s_2$, the curve $\gamma$ intersects at least one of $\rho_1$ or $\rho_2$. Following that curve to the point $p_0$ gives a train route from $p$ to $p_0$. Repeating the argument in the reverse direction produces a train route from $p_0$ back to $p$. Thus every point of $\tau$ lies on a closed train route. Hence $\tau$ is recurrent. In particular, $\tau$ admits a positive rational transverse measure $\mu$.

Let $\mu_1$ and $\mu_2$ be positive rational transverse measures on $\tau'$ determining slopes $s_1$ and $s_2$, respectively. Extend $\mu_1$ and $\mu_2$ to transverse measures on $\tau$ by assigning weight $0$ to branches in $\tau\setminus\tau'$.

Since $s_1\neq s_2$, the homology classes $[\mu_1]$ and $[\mu_2]$ are linearly independent in $H_1(T;\mathbb Q)\cong\mathbb Q^2$. Hence there exists a rational number $t$ such that \([\mu+t\mu_2]\) is a multiple of $[\mu_1]$ in $H_1(T;\mathbb Q)$.

Therefore, for a sufficiently small positive rational number $\varepsilon$, \(\mu_1+\varepsilon(\mu+t\mu_2)\) is a positive rational transverse measure on $\tau$ determining slope $s_1$. Since $\tau$ carries no null-homotopic closed curves, it fully carries a union of simple closed curves with slope $s_1$.
\end{proof}

\subsection{Local behavior near a crossing}

In this subsection, we analyze the local structure of the boundary train track of $B'$ near each crossing of the link diagram. 

Our first step is to select a point $x'$ in $B'\cap \partial \nu(L)$ corresponding to each arc cut by the crossings. Such an arc corresponds to a horizontal segment connecting two adjacent vertices in the brick diagram.

Recall that a point $x'\in B'$ is determined by a point $x\in B$ together with its relative height with respect to $F$. For each arc cut by the crossings, choose a point $p$ on the train track $\tau$ corresponding to this arc, and let
\[
(p,t)\in (\tau\times[0,3])\setminus \operatorname{int}(\nu(L))
\]
be the point with minimal $t$ subject to $t\ge h(p)$. By the construction of $B$, $(p,t)$ determines a point $x\in B\cap \partial \nu(L)$.

To specify the relative height of $x'\in B'\cap \partial \nu(L)$ with respect to $F$, it suffices to specify its relative height with respect to the output curve $\gamma_i$ passing through $p$ (if one exists). We use the following rule:

If $p$ lies on or below the initial row of Algorithm~\ref{alg:brick-curve}, then $x'\in B'\cap \partial \nu(L)$ is above the output curve $\gamma_i$ passing through $p$; 
otherwise, $x'\in B'\cap \partial \nu(L)$ is below the output curve $\gamma_i$ passing through $p$.

We next study how the points $x'\in B'\cap \partial \nu(L)$ are connected in the boundary train track near a crossing. If the corresponding arcs are connected by an under-strand, then they are connected directly together, as in the blackboard framing, without any switches. If the corresponding arcs are connected by an over-strand, the local configuration of the boundary train track, together with the position of the point $x'\in B'$, is one of the twelve cases shown in Figure~\ref{fig:overstrand-cases}. 

\begin{figure}[!htbp]
\centering

\newcommand{\casefig}[2]{
\begin{subfigure}{0.47\textwidth}
\centering
\includegraphics[width=\linewidth]{#1}
\caption{$(#2)$}
\end{subfigure}
}

\casefig{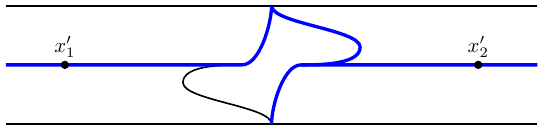}{X_1}\hfill
\casefig{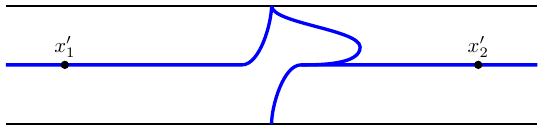}{X_2}

\vspace{3mm}
\casefig{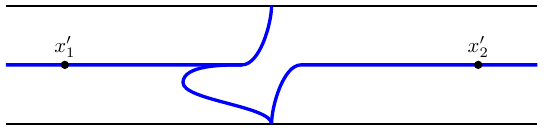}{X_3}\hfill
\casefig{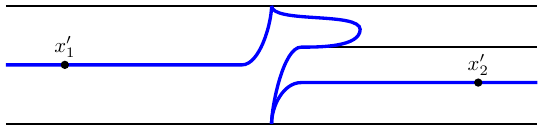}{X_4}

\vspace{3mm}
\casefig{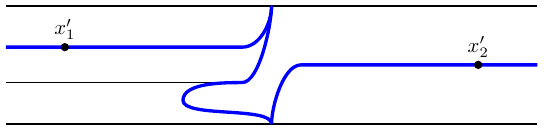}{X_5}\hfill
\casefig{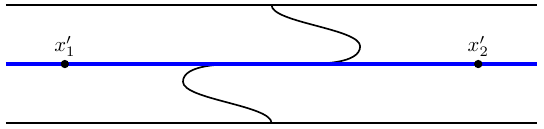}{Y_1}

\vspace{3mm}
\casefig{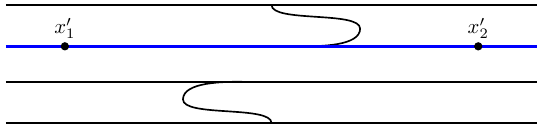}{Y_2}\hfill
\casefig{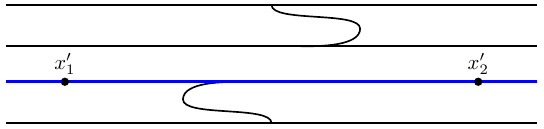}{Y_3}

\vspace{3mm}
\casefig{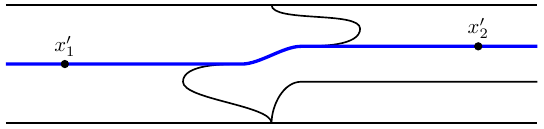}{Y_4}\hfill
\casefig{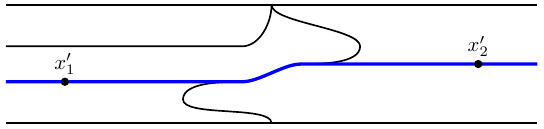}{Y_5}

\vspace{3mm}
\casefig{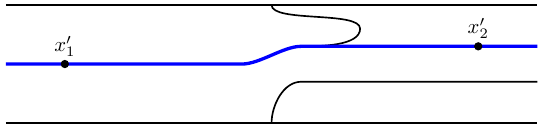}{Y_6}\hfill
\casefig{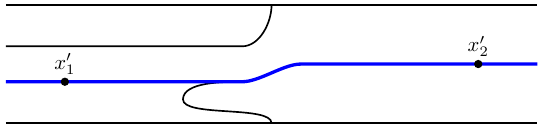}{Y_7}

\caption{Local behavior of boundary train track.}
\label{fig:overstrand-cases}
\end{figure}

If the output curves do not meet the vertical segment corresponding to a crossing, then the local configuration of the boundary train track is as in case~$(Y_1)$.

Each procedure may produce crossings for which the nearby boundary train track
is in one of the following cases.

\begin{itemize}
\item \textsc{CrossBrickStep}: $(X_2),(X_3),(Y_2),(Y_4)$;
\item \textsc{StepBelow}: $(X_1),(Y_2),(Y_4)$;
\item \textsc{StepAbove}: $(X_1),(Y_3),(Y_5)$;
\item \textsc{CloseAndStore}: $(Y_2),(Y_3)$;
\item \textsc{RestartBelow}: $(X_3),(X_5),(Y_6)$;
\item \textsc{RestartAbove}: $(X_2),(X_4),(Y_7)$.
\end{itemize}

We call a crossing Type~X if, at some step of Algorithm~\ref{alg:brick-curve}, the construction of $\gamma$ from the $A$-side (that is, the left side of $\gamma$ when the construction is on or below the row containing the initial point $q$, and the right side of $\gamma$ when the construction is above that row) either
\begin{enumerate}
    \item runs along the entire corresponding vertical segment, or
    \item enters or exits a brick through an endpoint of that vertical segment.
\end{enumerate}
Otherwise, the crossing is called Type~Y. Each application of any of the procedures \textsc{StepBelow}, \textsc{StepAbove}, \textsc{RestartBelow}, or \textsc{RestartAbove} produces one Type~X crossing. An application of \textsc{CrossBrickStep} produces two Type~X crossings. The total number of Type~X crossings equals the number of strands in the braid diagram. In Figure~\ref{fig:brick}, red dots mark Type~X crossings.

Near a Type~X crossing, the local configuration of the boundary train track is one of the cases $(X_1)$--$(X_5)$. In each case we can find train routes, shown as blue curves in Figure~\ref{fig:overstrand-cases}, with the following
properties.

\begin{lemma}\label{X-crossing-1}
For each Type~X crossing, the boundary train track contains a train route connecting the selected points $x_1'$ and $x_2'$ on the over-strand that introduces one additional twist relative to the blackboard framing.
\end{lemma}

\begin{lemma}\label{X-crossing-2}
For each Type~X crossing, the boundary train track contains a closed train route representing a meridian of the over-strand.
\end{lemma}

Near a Type~Y crossing, the local configuration of the boundary train track is one of the cases $(Y_1)$--$(Y_7)$. The blue curves in Figure~\ref{fig:overstrand-cases} again realize the corresponding train routes.

\begin{lemma}\label{Y-crossing}
For each Type~Y crossing, the boundary train track contains a train route connecting the selected points $x_1'$ and $x_2'$ on the over-strand that agrees with the blackboard framing.
\end{lemma}
\subsection{Slope realization}

Let $L_1,\ldots,L_n$ denote the components of $L$. 
To determine the realizable slope interval on each boundary torus, we first associate a vector to each $x\in (S^1\times I)\setminus \tau$.

For a point $x\in (S^1\times I)\setminus \tau$, choose a vertical segment starting above the first row and ending at $x$. Define
\[
\mathbf r(x)=(r_1,\ldots,r_n)\in \mathbb{Z}_{\ge 0}^n,
\]
where $r_k$ is the number of times this segment crosses strands of the component $L_k$. The value of $\mathbf r(x)$ is independent of the chosen vertical segment, and $\mathbf r(x)$ is locally constant on $(S^1\times I)\setminus \tau$. In particular, moving downward across a strand of $L_k$ increases the $k$-th coordinate by $1$.

Suppose that the brick in \textsc{CrossBrickStep} (line~\ref{line:CrossBrickStep}) lies between row $t$ and row $t+1$. Suppose that the left and right vertical edges of this brick correspond to the Type~X crossings $c_{t+1}$ and $c_{t}$, respectively. For $1\le i\le t$, let $c_i$ denote the Type~X crossing between the $i$-th row and the $(i+1)$-st row. For $t+2\le i\le m$ (where $m$ is the number of strands in the braid diagram), let $c_i$ denote the Type~X crossing between the $(i-1)$-st row and the $i$-th row.

Since $c_{t+1}$ and $c_t$ correspond to the left and right vertical edges of this brick, the brick to the left of $c_t$ is the same as the brick to the right of $c_{t+1}$.

For each $t+1\le i \le m-1$, $c_{i+1}$ is determined during the downward expansion phase by an application of \textsc{StepBelow} or \textsc{RestartBelow}. In either case, the brick below $c_i$ is the same as the brick to the right of $c_{i+1}$.

For each $1\le i \le t-1$, either $c_i$ is determined during the upward expansion phase by an application of \textsc{StepAbove} or \textsc{RestartAbove}, or both $c_i$ and $c_{i+1}$ are determined by \textsc{StepBelow} steps before the \textsc{CrossBrickStep} at line~\ref{line:CrossBrickStep}. In the second case, we have $d_{\mathrm{up}}(B,A)=2$ before applying \textsc{StepBelow} to determine $c_i$. In either case, the brick to the left of $c_i$ is the same as the brick above $c_{i+1}$.

Select points $x_0, x_1,\ldots, x_m\in (S^1\times I)\setminus \tau$ as follows:
\begin{enumerate}
    \item $x_0$ lies above the first row in the brick diagram;
    \item For each $1\le i \le t-1$, $x_i$ lies in the brick to the left of $c_i$;
    \item For each $t\le i\le m-1$, $x_i$ lies in the brick to the right of $c_i$;
    \item $x_m$ lies below the last row in the brick diagram.
\end{enumerate}
Then each $\mathbf r(x_{i+1})$ is obtained from $\mathbf r(x_i)$ by increasing the coordinate corresponding to the link component containing the over-strand at the crossing $c_i$ by $1$.

We are now ready to establish an interval of realizable slopes.

\begin{theorem}\label{realizable-slopes}
    For each link component $L_k$ of $L$, every slope in \((-\infty,\, 2g(L_k)-1)\cap \mathbb{Q}\) is realized by the boundary train track $B'\cap \partial \nu(L_k)$, where $g(L_k)$ denotes the genus of $L_k$.
\end{theorem}
\begin{proof}
    Suppose that the $k$-th coordinate of $\mathbf r(x_m)$ is $m_k$. Then the number of Type~X crossings whose over-strand lies in $L_k$ equals $m_k$. Since $L_k$ is the closure of a positive braid on $m_k$ strands, its genus is \(g(L_k)=\frac{N_k-m_k+1}{2},\) where $N_k$ denotes the number of crossings in the diagram $D$ whose over-strand and under-strand both lie in $L_k$.

   The slope of the blackboard framing on $\partial\nu(L_k)$ is $N_k$. Hence, by Lemmas~\ref{X-crossing-1} and~\ref{Y-crossing}, there exists a simple closed train route on $B'\cap \partial \nu(L_k)$ with slope
\[
N_k-m_k=2g(L_k)-1.
\]
By Lemma~\ref{X-crossing-2}, there also exists a simple closed train route on $B'\cap \partial \nu(L_k)$ with slope $\infty$. The union of these train routes forms a sub-train track $\tau'\subset B'\cap \partial \nu(L_k)$ which realizes every rational slope in $(-\infty,\,2g(L_k)-1)$.

Since $B'$ is co-orientable (see \Cref{cooriented-2}), the boundary train track $B'\cap \partial \nu(L_k)$ is also co-orientable. Since the torus $\partial \nu(L_k)$ is orientable, this boundary train track is orientable as well. By \Cref{bigon}, the complement of $B'\cap \partial \nu(L_k)$ in $\partial \nu(L_k)$ is a union of bigons. Therefore, by \Cref{subtrack-realize}, every slope in $(-\infty,\,2g(L_k)-1)\cap \mathbb{Q}$ is realized by $B'\cap \partial \nu(L_k)$.
\end{proof}
\section{Conclusion}\label{sec:6}
In this section, we prove the main theorem using the multi-torus-boundary version of \cite[Theorem~2.2]{li2003boundary}, as explicitly stated in \cite[Theorem~3.12]{santoro2024spaces} and \cite[Theorem~2.10]{kalelkar2014taut}.

\begin{theorem}[cf.\ \cite{li2003boundary}, Theorem~2.2]\label{essential-lamination-li}
Let $M$ be an irreducible and orientable $3$-manifold whose boundary is a union of incompressible tori $T_1, T_2, \ldots, T_n$. Suppose that $B$ is a laminar branched surface properly embedded in $M$ and $\partial M\setminus \partial B$ is a union of bigons. Then for any multislope $(s_1, \ldots, s_n) \in (\mathbb{Q}\cup\{\infty\})^n$ that can be realized by the train track $\partial B$, if $B$ does not carry a torus that bounds a solid torus in the manifold after Dehn filling along $(s_1, \ldots, s_n)$, then there exists a lamination $\mathcal{L}$ in $M$ fully carried by $B$ that intersects $\partial M$ in parallel simple closed curves of multislope $(s_1, \ldots, s_n)$.
\end{theorem}

We now prove the main theorem.

\begin{theorem}
Suppose $L$ is the closure of a non-split positive braid in $S^3$ with components $L_1,\ldots,L_n$, and assume that at least one component is not the unknot. Then the Dehn surgery along a multislope $(s_1,\ldots,s_n)\in\mathbb{Q}^n$ satisfying $s_i<2g(L_i)-1$ for all $i=1,\ldots,n$ yields a $3$-manifold admitting a co-oriented taut foliation.
\end{theorem}
\begin{proof}
    Since $L$ is non-split and non-trivial, $S^3\setminus  \operatorname{int}(\nu(L))$ is an irreducible, orientable $3$-manifold whose boundary is a union of incompressible tori. By \Cref{laminar}, $B'$ is a laminar branched surface properly embedded in $S^3 \setminus \operatorname{int}(\nu(L))$. By \Cref{bigon}, $\partial \nu(L)\setminus B'$ is a union of bigons. By \Cref{realizable-slopes}, any multislope $(s_1,\ldots,s_n)\in\mathbb{Q}^n$ satisfying $s_i<2g(L_i)-1$ for all $i=1,\ldots,n$ is realized by the boundary train track $B'\cap \partial \nu(L)$.
    
    By \Cref{transversal-2}, there exists a closed curve $\gamma$ in $S^3\setminus \operatorname{int}(\nu(L))$ positively transverse to $B'$ and intersecting every branch sector of $B'$. For any connected surface $S$ carried by $B'$, the curve $\gamma$ meets $S$ in a nonzero number of points, all with positive sign; hence $S$ cannot be separating. In particular, no torus carried by $B'$ in $S^3\setminus \operatorname{int}(\nu(L))$ becomes the boundary of a solid torus after any Dehn filling.

    By \Cref{essential-lamination-li}, there exists a lamination $\mathcal{L}$ in $S^3\setminus \operatorname{int}(\nu(L))$ fully carried by $B'$ that intersects $\partial \nu(L)$ in parallel simple closed curves of multislope $(s_1,\ldots,s_n)$. Since $\mathcal{L}$ is fully carried by $B'$, we may replace the fibered neighborhood $N(B')$ by a smaller one so that the horizontal boundary $\partial_h B'$ is contained in $\mathcal{L}$. The complement \((S^3\setminus \operatorname{int}(\nu(L)))\setminus \mathcal{L}\) decomposes as the union of \((S^3\setminus \operatorname{int}(\nu(L)))\setminus N(B')\) and \(N(B')\setminus \mathcal{L},\) glued along $\partial_v N(B') \setminus\mathcal{L}$.
    
By \Cref{disk-2}, \((S^3\setminus \operatorname{int}(\nu(L)))\setminus N(B')\) is a trivial open \(I\)-bundle over a disk. By \Cref{cooriented-2}, \(N(B')\setminus \mathcal{L}\) is a trivial open \(I\)-bundle over a co-oriented surface. Moreover, each component of \(\partial_v N(B')\setminus \mathcal{L}\) is a trivial open \(I\)-bundle over an interval. Since the gluing preserves the co-orientation, \((S^3\setminus \operatorname{int}(\nu(L)))\setminus \mathcal{L}\) also admits a trivial open \(I\)-bundle structure over a co-oriented surface. Together with the product foliation by horizontal surfaces, the lamination \(\mathcal{L}\) forms a co-oriented foliation \(\mathcal{F}\) of \(S^3\setminus \operatorname{int}(\nu(L))\). Its restriction to \(\partial \nu(L)\) is a product foliation by parallel simple closed curves of multislope \((s_1,\ldots,s_n)\). The foliation further extends over the Dehn filling with multislope \((s_1,\ldots,s_n)\) to a co-oriented foliation of the filled manifold.

By \Cref{transversal-2}, there exists a closed curve $\gamma$ in $S^3\setminus \operatorname{int}(\nu(L))$ positively transverse to $B'$ and intersecting every branch sector of $B'$. It follows by a standard argument that, after an isotopy, $\gamma$ is transverse to \(\mathcal{F}\) and intersects every leaf. Therefore, the filled manifold admits a co-oriented taut foliation.
\end{proof}

\bibliographystyle{alpha}
\bibliography{ref}
\end{document}